\documentclass[a4paper,11pt,leqno]{article}
\usepackage{amssymb, amsmath, amsthm, graphicx}
\usepackage{bbm}
\usepackage{xcolor}
\newtheorem{theorem}{Theorem}[section]
\newtheorem{lem}[theorem]{Lemma}
\newtheorem{cor}[theorem]{Corollary}
\newtheorem{prop}[theorem]{Proposition}
\newtheorem{defn}{Definition}[section]

\newtheorem{rem}{Remark}[section]

\numberwithin{equation}{section}

\newcommand{\sumtwo}[2]{\sum_{\substack{#1 \\ #2}}} 

\newcommand{\dis}{\displaystyle}

\renewcommand{\a}{\alpha}
\renewcommand{\b}{\beta}
\newcommand{\e}{\varepsilon}
\newcommand{\de}{\delta}
\newcommand{\fa}{\varphi}
\newcommand{\ga}{\gamma}
\renewcommand{\k}{\kappa}

\newcommand{\si}{\sigma}

\newcommand{\om}{\omega}
\newcommand{\De}{\Delta}
\newcommand{\Ga}{\Gamma}
\newcommand{\La}{\Lambda}

\newcommand{\lan}{\langle}
\newcommand{\ran}{\rangle}

\def\R{{\mathbb{R}}}

\def\Z{{\mathbb{Z}}}
\def\T{{\mathbb{T}}}

\allowdisplaybreaks

\title{Fast-reaction limit for Glauber-Kawasaki dynamics
\\ with two components}

\author{Anna De Masi, Tadahisa Funaki, Errico Presutti and Maria Eul\'alia Vares}
\date{\today}

\begin{document}
\maketitle

\begin{abstract}
\noindent
We consider the Kawasaki dynamics of two types of particles under a killing effect
on a $d$-dimensional square lattice. Particles move with possibly different
jump rates depending on their types.  The killing effect acts when particles of different
types meet at the same site. We show the existence of a limit under the diffusive space-time scaling and suitably growing killing rate: segregation of distinct types of particles does occur,  and the evolution of
the interface between the two distinct species is governed by the two-phase Stefan
problem. We apply the relative entropy method and combine it with some PDE techniques.
\end{abstract}

\section{Introduction}

The study of the fast-reaction limit in the reaction diffusion systems
goes back more than 20 years.  The motivation of this study comes from population dynamics
\cite{DHMP}, \cite{CDHMN}, \cite{IMMN}, mass-action kinetics chemistry
\cite{DDJ} and others.  Consider the system consisted of two types of species, say
$A$ and $B$, and assume each of them moves by diffusion with rates $d_1$ and $d_2$,
respectively.  When distinct species meet, they kill each other with high rate $K$.
This problem is formulated in PDE terminology as the system of equations
for densities $u_1(t,r)$ and $u_2(t,r)$ of species $A$ and $B$, respectively, written as
\begin{equation}  \label{eq:1.1}
\partial_t u_i(t) = d_i \De u_i(t)- K u_1(t)u_2(t), \quad i=1,2.
\end{equation}
Several papers including those cited above studied the limit as $K\to\infty$
of the solutions $u_i(t,r)$ of the system \eqref{eq:1.1} or its extensions,
that is, the limit as the killing rate of distinct species gets large.
This is called the fast-reaction limit.  It is known that the segregation
of two species occurs in the limit and the interface separating two distinct species
evolves according to the two-phase Stefan free boundary problem.

In the present paper, we formulate the problem at the original level of
species, i.e., at the underlying microscopic level,
and model it as a system with two distinct types of
particles.  Under a diffusive space-time scaling combined with the limit as $K\to\infty$
taken properly, we prove that the segregation of species occurs at macroscopic level and
derive the Stefan problem directly from our microscopic system.

The proof is divided into two parts and given as a combination of the techniques of
the hydrodynamic limit and the fast-reaction limit. In the first part, which is probabilistic,
we consider the relative entropy of the real system with respect to the
local equilibria defined as a product measure with mean changing in space and time
chosen according to the discretized hydrodynamic equation,
which is a discrete version of \eqref{eq:1.1}.  Then, we show that the
relative entropy behaves as
a small order of the total volume of the system.   This proves that
the macroscopic density profile of the system is close to the solution
of the discretized hydrodynamic equation.  We take product measures as
local equilibria, since those with constant means are global equilibria of
the Kawasaki dynamics. In the second part of the paper, we apply PDE results to
analyze the discrete equation and derive the Stefan problem from it.

\subsection{Model}
Let $\T_N^d :=(\Z/N\Z)^d\equiv \{1,2,\ldots,N\}^d$ be the $d$-dimensional
periodic square lattice of size $N$ and consider a system that consists of particles
of two distinct types.  The configuration space is $\mathcal{X}_N^2
=\mathcal{X}_N\times\mathcal{X}_N$, where $\mathcal{X}_N= \{0,1\}^{\T_N^d}$. Its
elements are denoted by $\tilde\si \equiv (\si_1, \si_2)
= \big(\{\si_{1,x}\}_{x\in \T_N^d},
\{\si_{2,x}\}_{x\in \T_N^d}\big)$.  The generator of the Kawasaki
dynamics of particles of a single type is defined by
\begin{align*}
(L_0^1 f)(\si)
&  = \frac12 \sum_{x, y\in\T_N^d:|x-y|=1}
\left\{  f(\si^{x,y})  -f(\si) \right\}, \quad \si = \{\si_x\}_{x\in \T_N^d}\in\mathcal{X}_N,
\end{align*}
for functions $f\colon\mathcal{X}_N\to \mathbb{R}$, and
where $\si^{x,y}\in  \mathcal{X}_N$ is defined from
$\si\in  \mathcal{X}_N$ as

$$
    (\sigma^{x,y})_z = \left\{
        \begin{array}{lll}
            \si_y & \hbox{if $z = x$,}\\
            \si_x & \hbox{if $z = y$,} \\
             \si_z & \hbox{\text{ otherwise}.}\\
        \end{array}
    \right.
$$
The generator of the two-component system
is given by $L_N = N^2L_0+K L_G$ with $K=K(N)$, where
\begin{align*}
& (L_0f)(\si_1,\si_2) = d_1 (L_0^1 f(\cdot,\si_2))(\si_1)
+ d_2 (L_0^1 f(\si_1,\cdot))(\si_2), \\
& (L_Gf)(\si_1,\si_2) = \sum_{x\in\T_N^d} \si_{1,x}\si_{2,x}
\left\{  f(\si_1^x,\si_2^x)  -f(\si_1,\si_2) \right\},
\end{align*}
for functions $f\colon\mathcal{X}_N^2\to \mathbb{R}$, and $d_1, d_2 >0$.
Here $\si^x\in  \mathcal{X}_N$ is defined from
$\si\in  \mathcal{X}_N$ as
$$
    \sigma^x_z = \left\{
        \begin{array}{ll}
            1-\si_x & \hbox{if $z = x$,}\\
            \si_z & \hbox{if $z \neq x$.}
        \end{array}
    \right.
$$
Let $\bar\si^N(t) = (\si_1^N(t),\si_2^N(t)) \equiv (\si_{1,x}^N(t),
\si_{2,x}^N(t))_{x\in \T_N^d}$ be the Markov process on $\mathcal{X}_N^2$
generated by $L_N$.  The macroscopic empirical measures on
$\T^d:= (\R/\Z)^d \equiv [0,1)^d$ associated with a configuration
$\bar\si = (\si_1,\si_2) \in \mathcal{X}_N^2$ are defined by
$$
\a_i^N(dr;\bar\si) = \frac1{N^d} \sum_{x\in\T_N^d} \si_{i,x}
\de_{\frac{x}N}(dr), \quad r \in \T^d, \; i=1,2.
$$
The goal is to study the limit of the macroscopic empirical measures of
the process $\bar\si^N(t)$ as $N\to\infty$, with properly scaled $K(N)$.

\subsection{Main result}

We first summarize our assumptions on the initial distribution of
$\bar\si^N(0)$.
\begin{itemize}
\item[(A1)] Let $u_i^N(0,\cdot) =\{u_i^N(0,x)\}_{x\in \T_N^d}, i=1,2$ be given and
satisfy two bounds
$$
e^{-c_1K} \le u_i^N(0,x) \le c_2 \quad \text{ and } \quad
|\nabla^N u_i^N(0,x)|\le C_0 K,\mathbb{}
$$
for every $i=1,2$, $x\in \T_N^d$ with $c_1>0, 0<c_2<1, C_0>0$, where $\nabla^N$
is defined as
\begin{equation}  \label{eq:nablaN}
\nabla^N u(x) = \left\{ N(u(x+e_i)-u(x))\right\}_{i=1}^d,
\end{equation}
and $e_i\in \Z^d$ are the unit vectors in the $i$th positive direction.
\item[(A2)] Let $\mu_0^N$ be the distribution of $\tilde\si^N(0)$ on
$\mathcal{X}_N^2$ and let $\nu_0^N = \nu_{u_1^N(0,\cdot),u_2^N(0,\cdot)}$
be the Bernoulli measure on $\mathcal{X}_N^2$ with means
$u_1^N(0,\cdot), u_2^N(0,\cdot)$ given as above.
We assume the relative entropy defined in \eqref{eq:2.re} satisfies
$H(\mu_0^N|\nu_0^N) = O(N^{d-\de_0})$ as $N\to\infty$ with some
$\de_0>0$.
\item[(A3)]  We assume $u_i^N(0,r), r\in \T^d$
defined from $u^N_i(0,x)$ through \eqref{eq:4.10} converge to some $u_i(0,r)$
weakly in $L^2(\T^d)$ as $N\to\infty$, for $i=1,2$, respectively.
\item[(A4)$_\de$] $K=K(N)$ satisfies
 $1\le K(N)\le \de (\log N)^{1/2}$  and $K(N)\to\infty$ as $N\to \infty$.
\end{itemize}

Our main theorem is formulated as follows.

\begin{theorem}  \label{thm:1.1}
We assume the four conditions (A1)-(A3),(A4)$_\de$ with $\de>0$ chosen
sufficiently small depending on $T>0$.  Then, we have the following. \\
{\rm (1)}  The macroscopic empirical measures $\a_i^N(t,dr):=\a_i^N(dr; \bar\si^N(t))$
of the process $\bar\si^N(t)$  converge
to $u_i(t,r)dr$, respectively, for $i=1,2$, that is
$$
\lim_{N\to\infty}P(|\lan\a_i^N(t),\fa\ran - \lan u_i(t),\fa\ran|>\e)=0,
\quad i=1, 2,
$$
for every $\e>0$, $t\in [0,T]$  and $\fa\in C^\infty(\T^d)$,
and $u_1(t,r) u_2(t,r)=0$ a.e.\ $r$ holds, where $\lan \a,\fa\ran$ and
$\lan u,\fa\ran$ stand for the integrals over $\T^d$. \\
{\rm (2)} $w(t,r) := u_1(t,r)-u_2(t,r)$ is the unique weak solution of
\begin{equation} \label{eq:P}
\left\{\begin{aligned}
& \partial_t w = \De \mathcal{D}(w), \quad \text{ in } \T^d, \\
& w(0,r)=u_1(0,r)-u_2(0,r),
\end{aligned} \right.
\end{equation}
where $\De$ is the Laplacian on $\T^d$, and
$\mathcal{D}(s) = d_1s, s\ge 0$ and $=d_2s, s<0$.
\end{theorem}

The weak solution of \eqref{eq:P} is defined as follows.

\begin{defn}  \label{def1.1}
We call $w=w(t,r)$ a weak solution of \eqref{eq:P} if
it satisfies
\begin{itemize}
\item[(i)]  $w \in L^\infty([0,T]\times \T^d)$  for every $T>0$;
\item[(ii)] For all $T>0$ and $\psi(t,x) \in C^{1,2}([0,T]\times \T^d)$
such that $\psi(T,r)=0$, we have
$$
\int_0^T\int_{\T^d} (w\partial_t \psi+\mathcal{D}(w)\De\psi)dr dt
= - \int_{\T^d} w_0\psi(0,r)dr.
$$
\end{itemize}
\end{defn}

The uniqueness of the weak solution of \eqref{eq:P} is
shown in \cite{CDHMN}, Corollary 3.8.  As pointed out in \cite{DHMP},
\eqref{eq:P} is the weak formulation of the following two-phase
Stefan problem for $u_1$ and $u_2$:
\begin{equation} \label{eq:Stefan}
\left\{\begin{aligned}
& \partial_t u_1 = d_1 \De u_1, \quad \text{ on }
D_1(t) = \{r\in \T^d; u_1(t,r)>0, \; u_2(t,r)=0\}, \\
& \partial_t u_2 = d_2 \De u_2, \quad \text{ on }
D_2(t) = \{r\in \T^d; u_1(t,r)=0, \; u_2(t,r)>0\}, \\
& u_1=u_2 =0, \qquad \text{ on } \Ga(t):= \partial D_1(t) = \partial D_2(t),\\
& d_1 \partial_n u_1 = - d_2 \partial_nu_2,\quad \text{ on } \Ga(t),
\end{aligned} \right.
\end{equation}
where $n$ is the unit normal vector on $\Ga(t)$ directed to $D_1(t)$.
Indeed, if the system \eqref{eq:Stefan} has a smooth solution, that is, if $\Ga(t)$ is
smooth, $u_i(t,r), i=1,2,$ are smooth in $D_i(t)$ and continuous on $\T^d$, then
it determines a weak solution.

The proof of Theorem \ref{thm:1.1} is divided into two parts
as we mentioned above.  The main task is to show that the relative
entropy of our system compared with the local equilibria defined
through the discretized hydrodynamic equation \eqref{eq:HD-discre}
behaves as $o(N^d)$, namely, the relative entropy per
volume tends to $0$ as $N\to\infty$.  This is formulated in
Theorem \ref{thm:EstHent} and shown in Sections \ref{sec:2} and \ref{sec:3}.
Once this is shown, one can prove that the macroscopic empirical measures
$\a_i^N(t)$ is close to the solution of \eqref{eq:HD-discre},
see Section \ref{sec:4}.  In the last Section \ref{sec:5}, we show
that the solution of \eqref{eq:HD-discre} converges to the weak
solution of \eqref{eq:P}.

A related model with instantaneous annihilation was studied by Funaki \cite{F99} and
the same equation \eqref{eq:P} was derived in the limit.  Briefly saying, $1\ll K \ll N$ in our model,
while $1\ll N \ll K=\infty$ in \cite{F99}.  Sasada \cite{Sa} considered the model with
non-instantaneous annihilation together with creation of two distinct types of particles.

\section{Relative entropy method}  \label{sec:2}

The relative entropy of two probability measures $\mu$ and $\nu$ on $\mathcal{X}_N^2$ is defined as
		\begin{equation}
		\label{eq:2.re}
H(\mu|\nu) := \int_{\mathcal{X}_N^2} \frac{d\mu}{d\nu} \log \frac{d\mu}{d\nu} \cdot d\nu.
		\end{equation}
For a probability measure $\nu$ on $\mathcal{X}_N^2$,  the
Dirichlet form $\mathcal{D}(f;\nu)$, $f$:$\mathcal{X}_N^2 \to \R$, associated to the generator $L_0$ is defined as
$$
\mathcal{D}(f;\nu) = \frac14\sumtwo{x,y\in\T_N^d}{ |x-y|=1}
\int_{\mathcal{X}_N^2} \left[d_1\{f(\si_1^{x,y},\si_2)-f(\si_1,\si_2)\}^2
+d_2\{f(\si_1,\si_2^{x,y})-f(\si_1,\si_2)\}^2 \right]d\nu.
$$
Let $\mu_t=\mu_t^N$ be the law of $(\si_1^N(t),\si_2^N(t))$ generated by
$L_N = N^2L_0+KL_G$ on $\mathcal{X}_N^2$.

We have the following estimate on the time derivative of
the relative entropy. See \cite{F18}, \cite{FT} for the proof.

\begin{prop} \label{prop:RE}
For any probability measures $\{\nu_t\}$ and $m$ on
$\mathcal{X}_N^2$ both with full supports in
$\mathcal{X}_N^2$, we have
\begin{equation} \label{eq:dH}
\frac{d}{dt} H(\mu_t|\nu_t)
\le - 2 N^2 \mathcal{D}\left(\sqrt{\frac{d\mu_t}{d\nu_t}}; \nu_t\right)
+ \int_{\mathcal{X}_N^2} \left( L_N^{*,\nu_t} 1 - \partial_t \log \psi_t \right) d\mu_t.
\end{equation}
where $L_N^{*,\nu_t}$  is the adjoint of $L_N$ on $L^2(\nu_t)$ and
$\dis{\psi_t := \frac{d\nu_t}{dm}}$.
\end{prop}

This estimate was first used by Guo, Papanicolaou and Varadhan
taking $\nu_t$ to be a global equilibrium which is independent of $t$
and then by Yau dropping the negative Dirichlet form term, see \cite{F18}.
Then Jara and Menezes introduced \eqref{eq:dH} as a combination of these two estimates,
cf.\ \cite{JM}.

\medskip

We use \eqref{eq:dH} with the following Bernoulli measures $\nu_t$. Let $u_i^N(t) = \{u_i(t,x)=u_i^N(t,x)\}_{x\in \T_N^d}$,
$i=1,2$ be the solution of the system of the discretized hydrodynamic equation:
\begin{equation} \label{eq:HD-discre}
\partial_t u_i^N(t,x) = d_i \De^N u_i^N(t,x) -K u_1^N(t,x)u_2^N(t,x), \qquad  i=1,2
\end{equation}
 where $\De^N = N^2 \De$ and
$$
(\De u)(x) = \sum_{y\in \T_N^d:|y-x|=1} \big(u(y)-u(x)\big).
$$
Note that \eqref{eq:HD-discre} is a discrete version of \eqref{eq:1.1}.
We define
$\nu_t=\nu_{u_1(t,\cdot),u_2(t,\cdot)}$,   where we denote by $\nu_{u_1(\cdot),u_2(\cdot)}$
for $u_i(\cdot)=\{u_i(x)\}_{x\in\T_N^d}, i=1,2$ the Bernoulli measure
on $\mathcal{X}_N^2$ such that
$$
\nu_{u_1(\cdot),u_2(\cdot)}(\si_{i,x}=1) = u_i(x),
$$
for every $x \in \T_N^d$ and $i=1,2$.

The main result in the probabilistic part is the following Theorem.

\begin{theorem} \label{thm:EstHent}
Assume the initial distribution verifies the Hypothesis of Theorem \ref{thm:1.1}.
Then, we have
$$
H(\mu_t^N|\nu_t^N) = o(N^d), \quad t\in [0,T],
$$
as $N\to\infty$.  
\end{theorem}

The proof of this theorem needs some preliminary results proved in the following subsections.

\subsection{Calculation of the second term in \eqref{eq:dH}}

We define the normalized variables $\om_{i,x,t}$ by
		\begin{equation}
		\label{2.3b}
\om_{i,x,t} = \frac{\bar\si_{i,x}(t)}{\chi(u_i(t,x))},  \qquad \bar{\si}_{i,x}(t) = \si_{i,x}(t)-u_i(x,t), \quad i=1, 2,
	\end{equation}
where $\chi(u)=u(1-u)$ for $u\in (0,1)$.

In this subsection we prove the following proposition.
\begin{prop} \label{prop:LG1}

	\begin{equation}
	\label{eq:1.2}L_N^{*,\nu_t}1 - \partial_t \log\psi_t
= V_1(t)+V_2(t)+V(t)\end{equation}
with
	\begin{eqnarray}
&&	V_i(t)= -\frac{d_iN^2}2 \sum_{x,y\in\T_N^d:|x-y|=1}  (u_i(t,y)-u_i(t,x))^2 \om_{i,x,t} \om_{i,y,t},\qquad i=1,2  \label{eq:1.2a}
\\&& V(t) =K \sum_{x\in \T_N^d} \big(u_1(t,x)+u_2(t,x)-1\big)
u_1(t,x)u_2(t,x) \om_{1,x,t}\om_{2,x,t}.
 \label{eq:1.2b}
	\end{eqnarray}

\end{prop}


\begin{proof} We first compute $L_N^{*,\nu_t}g$ for a generic function  $g$
on $\mathcal{X}_N^2$, calling $\nu= \nu_{u_1(\cdot),u_2(\cdot)}$.

For $L_G$ we have
\begin{align*}
\int g L_G f d \nu
= \sum_{\si_1,\si_2} g(\si_1,\si_2) \sum_{x\in\T_N^d}
\si_{1,x}\si_{2,x} \left\{  f(\si_1^x,\si_2^x)  -f(\si_1,\si_2) \right\}
\nu(\si_1,\si_2).
\end{align*}
By making change of variables $\eta_1=\si_1^x$ and
$\eta_2=\si_2^x$, the sum containing $f(\si_1^x,\si_2^x)$
can be rewritten as
\begin{align*}
\sum_{\eta_1,\eta_2} g(\eta_1^x,\eta_2^x) \sum_{x\in\T_N^d}
(1-\eta_{1,x})(1-\eta_{2,x}) f(\eta_1,\eta_2)
\nu(\eta_1^x,\eta_2^x).
\end{align*}
Next observe that for $\eta_1, \eta_2$ satisfying $\eta_{1,x}=\eta_{2,x}=0$,
$$
\nu(\eta_1^x,\eta_2^x) =
\frac{u_1(x)u_2(x)}{(1-u_1(x))(1-u_2(x))}
\nu(\eta_1,\eta_2),
$$
Thus
\begin{align*}
 L_G^{*,\nu} g(\si_1,\si_2)
= \sum_{x\in \T_N^d} & \Bigg\{ \frac{u_1(x)u_2(x)}{(1-u_1(x))(1-u_2(x))}
(1-\si_{1,x})(1-\si_{2,x}) g(\si_1^x,\si_2^x) \\
& \hskip 60mm
- \si_{1,x}\si_{2,x} g(\si_1,\si_2) \Bigg\}.
\end{align*}
Using the above equality with $g\equiv 1$ and writing $\bar\si_{i,x} = \om_{i,x,t} \chi(u_i(\cdot))$ (recall \eqref{2.3b})
we get
%
%
\begin{align*}
K L_G^{*,\nu_{u_1(\cdot),u_2(\cdot)}} 1
= & - K \sum_{x\in \T_N^d} u_1(x)u_2(x) \big(\om_{1,x}+\om_{2,x}\big)\\
& + K \sum_{x\in \T_N^d} \big(u_1(x)+u_2(x)-1\big)
u_1(x)u_2(x) \om_{1,x}\om_{2,x}.
\end{align*}

For the Kawasaki part, from the computation in \cite{F18} or \cite{FT}, we obtain

\begin{align*}
L_0^{*,\nu}1 = & -\frac{d_1}2 \sumtwo{x,y\in\T_N^d}{|x-y|=1}  (u_1(y)-u_1(x))^2 \om_{1,x} \om_{1,y}
+ d_1\sum_{x\in\T_N^d} (\De u_1)(x) \om_{1,x} \\
& -\frac{d_2}2 \sumtwo{x,y\in\T_N^d}{|x-y|=1}  (u_2(y)-u_2(x))^2
\om_{2,x} \om_{2,y}
+ d_2\sum_{x\in\T_N^d} (\De u_2)(x) \om_{2,x}.
\end{align*}
We next observe that
\begin{equation*}
\partial_t \log\psi_t(\eta) = \sum_{x\in\T_N^d} \partial_t u_1(t,x) \om_{1,x,t}
+ \sum_{x\in\T_N^d} \partial_t u_2(t,x) \om_{2,x,t}.
\end{equation*}
this equality is proved similarly to \cite{F18} or \cite{FT}.
By using \eqref{eq:HD-discre} the linear terms in $\om$ cancel and we finally obtain \eqref{eq:1.2}. \end{proof}


\subsection{Estimates on the solution of \eqref{eq:HD-discre}}
\label{subsec:2.2}

Let $(u_1(t),u_2(t))=(u_1^N(t,x),u_2^N(t,x))$ be the solution of the discretized hydrodynamic equation
\eqref{eq:HD-discre}.  We derive estimates on $(u_1(t),u_2(t))$ and their gradients.
First two lemmas, especially taking $c=c_2<1$ with $c_2$ in (A1),  are useful to estimate
$1/\chi(u_i^N(t,x))$ appearing in the definition of $\om_{i,x,t}$ from above.

\begin{lem}  \label{lem:3.1}
If the initial values satisfy $0\le u_i^N(0,x)\le c$ for every $x \in \T_N^d$ and $i=1,2$
with some $c>0$, we have
$$
0\le u_i^N(t,x) \le c,
$$
for every $t\ge 0$, $x \in \T_N^d$ and $i=1,2$.
\end{lem}

\begin{proof}
One can apply the maximum principle in our discrete setting, cf.\ \cite{CDHMN}, Lemma 2.1.
Also, a similar argument to the proof of the next lemma works.
\end{proof}

\begin{lem}  \label{lem:3.2}
If the initial values satisfy $0< u_0\le u_i^N(0,x) \le 1$ for every $x \in \T_N^d$ and $i=1,2$, we have
$$
u_i^N(t,x) \ge \underbar u(t) := u_0e^{-Kt},
$$
for every $t\ge 0$, $x \in \T_N^d$ and $i=1,2$.
\end{lem}

\begin{proof}
From \eqref{eq:HD-discre} and $u_2(t,x) \le 1$, since $\underbar u(t)$ satisfies
$\partial_t \underbar u(t) = -K\underbar u(t)$, we have
\begin{align*}
\partial_t \big(u_1(t,x) -\underbar u(t)\big)
&= d_1 \De^N \big(u_1(t,x)- \underbar u(t)\big)
 -K \big( u_1(t,x)u_2(t,x) -\underbar u(t)\big) \\
& \ge d_1 \De^N \big(u_1(t,x)- \underbar u(t)\big)
 -K \big( u_1(t,x) -\underbar u(t)\big).
\end{align*}
Assume that $u_1(s,y)> \underbar u(s)$ holds for $0\le s < t$ and every $y\in \T_N^d$, and at some $x$
and $t$, $u_1(t,x)= \underbar u(t)$ holds.  Then, $\De^N \big(u_1(t,x)- \underbar u(t)\big)\ge 0$
and $-K \big( u_1(t,x) -\underbar u(t)\big) =0$.  Therefore,
$\partial_t \big(u_1(t,x) -\underbar u(t)\big)\ge 0$.  This means that $u_1(t,x) -\underbar u(t)$
is increasing and $u_1(t,x)$ can not be below $\underbar u(t)$.  Same argument works
for $u_2(t,x)$.
\end{proof}

Let $p^N(t,x,y)$ be the discrete heat kernel corresponding to
$\De^N$ on $\T_N^d$.  Then, we have the following estimate, which is
global in $t$.

\begin{lem} \label{lem:2.1}
There exist $C, c>0$ such that
$$
|\nabla^N p^N(t,x,y)|
\le \frac{C}{\sqrt{t}} p^N(ct,x,y), \quad t>0,
$$
where $\nabla^N$ is defined by \eqref{eq:nablaN}.
\end{lem}

\begin{proof}
Let $p(t,x,y)$ be the heat kernel corresponding to the discrete Laplacian
$\De$ on $\Z^d$.  Then, we have the estimate
$$
|\nabla p(t,x,y)|
\le \frac{C}{\sqrt{1\vee t}} p(ct,x,y), \quad t >0, \; x, y \in \Z^d,
$$
with some constants $C, c>0$, independent of $t$ and $x, y$, where $\nabla=\nabla^1$.
This should be well-known, but we refer to \cite{DD} Theorem 1.1 (1.4)
which discusses general case with random coefficients, see also \cite{SZ}.
Then, since
$$
p^N(t,x,y) = \sum_{k\in (N\Z)^d} p(N^2t,x,y+k),
$$
the result follows.
\end{proof}

We have the following estimate, though it might not be the best possible one.

\begin{prop}  \label{prop:2.2}
The gradients of the solution of \eqref{eq:HD-discre} are estimated as
$$
|\nabla^N u_i^N(t,x)| \le K(C_0+C \sqrt{t}), \quad t >0, \; i=1, 2,
$$
if $|\nabla^N u_i^N(0,x)| \le C_0 K$ holds.
\end{prop}

\begin{proof}
From Duhamel's formula, we have
$$
u_i^N(t,x) = \sum_{y\in\T_N^d} u_i^N(0,y)p^N(d_it,x,y)
-K \int_0^t ds \sum_{y\in\T_N^d} u_1^N(s,y)u_2^N(s,y) p^N(d_i(t-s),x,y).
$$
By noting the symmetry of $p^N$ in $(x,y)$ and
$0\le u_1^N(s,x), u_2^N(s,x) \le 1$, this shows
$$
|\nabla^N u_i^N(t,x)| \le \sum_{y\in\T_N^d} |\nabla^N u_i^N(0,y)|p^N(d_it,x,y)
+K \int_0^t ds \sum_{y\in\T_N^d} |\nabla^N p^N(d_i(t-s),x,y)|.
$$
Thus, from Lemma \ref{lem:2.1}, we obtain the desired estimate.
\end{proof}

\subsection{Proof of Theorem \ref{thm:EstHent}}

{\bf {Notation:}} We simply denote $\mu_t=\mu_t^N, \nu_t = \nu_t^N$
and set $f_t\equiv f_t^N:= \frac{d\mu_t^N}{d\nu_t^N}$.

Recalling Proposition \ref{prop:LG1}, and using the estimates of subsection \ref{subsec:2.2}, in Section \ref{sec:3} we prove the following Theorem.

\begin{theorem} \label{thm:V3}
For $\alpha$ and $\kappa>0$ small, there is $C_{\a,\k} >0$ so that
 	\begin{equation}
	\label{2.9}
	\int_{\mathcal{X}_N^2} V(t) d\mu_t\le \a N^2 \mathcal{D}(\sqrt{f_t};\nu_t)
+ C_{\a,\k} K H(\mu_t|\nu_t) + N^{d-1+\k},
	\end{equation}
and also
\begin{equation}
	\label{2.10}
	\int_{\mathcal{X}_N^2} [V_1(t)+V_2(t)] d\mu_t\le \a N^2 \mathcal{D}(\sqrt{f_t};\nu_t)
+ C_{\a,\k} K^2 H(\mu_t|\nu_t) + N^{d-1+\k},
	\end{equation}
the  term
$N^{d-1+\k}$ is replaced by $N^{\frac12+\k}$ when $d=1$.

\end{theorem}

\medskip
By using Proposition \ref{prop:RE}, \eqref{eq:1.2} and the above Theorem,
we obtain
$$
\frac{d}{dt} H(\mu_t|\nu_t) \le C K^2 H(\mu_t|\nu_t) +  O(N^{d-\de_1}),
$$
with $0<\de_1<1$.  We have chosen $\a\in (0,1)$ so that the terms of
positive Dirichlet forms are absorbed by the negative Dirichlet form in
\eqref{eq:dH}.  Thus, Gronwall's inequality shows
$$
H(\mu_t|\nu_t) \le \left( H(\mu_0|\nu_0)+ t O(N^{d-\de_1}) \right) e^{CK^2t}.
$$
Noting $e^{CK^2t} \le N^{Ct\de^2}$ from $1\le K=K(N) \le \de (\log N)^{1/2}$
and $H(\mu_0|\nu_0) = O(N^{d-\de_0})$ by the
assumption, this concludes the proof of Theorem \ref{thm:EstHent}, if
$\de=\de_T>0$ is small enough such that $CT\de^2 < \de_0\wedge \de_1$.
\qed


\section{Proof of Theorem \ref{thm:V3}}
\label{sec:3}

We split the proof in two subsections.


\subsection{Proof of \eqref{2.9}}
\label{sec:4.1}

We omit the dependence on $t$ and define
$$
V := K \sum_{x\in\T_N^d} \tilde\om_{1,x} \om_{2,x},
$$
where  $\tilde\om_{1,x} = \big(u_1(x)+u_2(x)-1\big)
u_1(x)u_2(x) \om_{1,x}$.

The first step is to replace $V$ by its local sample average $V^\ell$
defined by
\begin{align}  \label{eq:4.VL}
V^\ell := K \sum_{x\in\T_N^d} \overleftarrow{(\tilde\om_1)}_{x,\ell}
\overrightarrow{(\om_2)}_{x,\ell},
\end{align}
where
$$
\overrightarrow{g}_{x,\ell} := \frac1{|\La_\ell|} \sum_{y\in\La_\ell}g_{x+y},
\qquad
\overleftarrow{g}_{x,\ell} := \frac1{|\La_\ell|} \sum_{y\in\La_\ell}g_{x-y},
$$
for a function $g= \{g_x(\si_1,\si_2)\}$ and $\La_\ell = [0,\ell-1]^d \cap \Z^d$.

\begin{prop}  \label{prop:2.9}
We assume the conditions of Theorem \ref{thm:EstHent}, in particular,
we take $\de>0$ sufficiently small.  Let $\nu=\nu_{u_1(\cdot),u_2(\cdot)}$,
$d\mu=fd\nu$ (recall we omit $t$) and we choose $\ell = N^{\frac1d-\k'}$
with $\k' (=\k/d)>0$ when $d\ge 2$ and $\ell = N^{\frac12-\k}$ when $d=1$,
with small $\k>0$.  Then the cost of the replacement is estimated as
\begin{align}  \label{eq:2.11}
\int(V-V^\ell)f d\nu \le \a N^2 \mathcal{D}(\sqrt{f};\nu)
+ C_{\a,\k} \left( H(\mu|\nu) + N^{d-1+\k}\right),
\end{align}
for every $\a, \k>0$ with some $C_{\a,\k}>0$ when $d\ge 2$ and the last
$N^{d-1+\k}$ is replaced by $N^{\frac12+\k}$ when $d=1$.
\end{prop}

The first tool to show this proposition is the flow lemma for the telescopic sum.  We call
$\Phi = \{\Phi(x,y)\}_{b=\{x,y\}\in G^*}$  a flow on a finite set $G$
connecting two probability measures $p$ and $q$ on $G$ if $\Phi(x,y)
= -\Phi(y,x)$ hold for all $\{x,y\}\in G^*$ and $\sum_{z\in G}\Phi(x,z) = p(x)-q(x)$
hold for all $x\in G$, where $G^*$ is the set of all bonds in $G$.
The following lemma is found in Appendix G of \cite{JM}, see also \cite{F18}, \cite{FT}.

\begin{lem} \label{lem:flow} (Flow lemma) There exists a flow
$\Phi^\ell$ on $\La_{2\ell} := \{0,1,\ldots,2\ell-1\}^d$ connecting
$\de_0$ and $q_\ell := p_\ell* p_\ell$,
$p_\ell (x)= \frac1{|\La_\ell|} 1_{\La_\ell}(x)$, such that
$$
\sum_{x\in \La_{2\ell-1}} \sum_{j=1}^d \Phi^\ell(x,x+e_j)^2
\le C_d g_d(\ell),
$$
where $e_j$ is a unit vector to $j$th positive direction, and
 $g_d(\ell) = \ell$ when $d=1$, $\log\ell$ when $d=2$ and $1$ when
$d\ge 3$.
\end{lem}

\begin{rem}
{\rm (1)} When $d=1$, the flow $\Phi^\ell$ on $\La_{\ell+1} = \{0,1,\ldots,\ell\}$ connecting
$\de_0$ and $p_\ell(x) = \frac1\ell 1_{\{1,\ldots,\ell\}}(x)$
is given by $\Phi^\ell(x,x+1) = \frac{\ell-x}\ell, 0\le x \le \ell-1$.  Indeed, the condition
on $\Phi^\ell$ is
$$
\Phi^\ell(x,x+1)+\Phi^\ell(x,x-1) = \de_0(x)-p_\ell(x), \quad x \in \La_\ell,
$$
with $\Phi^\ell(\ell,\ell+1) = \Phi^\ell(0,-1)=0$.  Or equivalently, recalling that
$\Phi^\ell(x,x-1)=- \Phi^\ell(x-1,x)$ and setting $\tilde\Phi(x) := \Phi^\ell(x,x+1)$,
the condition is
\begin{align*}
& \nabla\tilde\Phi(x) \left(= \tilde\Phi(x) - \tilde\Phi(x-1) \right) = -\frac1\ell,
\quad 1 \le x \le \ell, \\
& \tilde\Phi(0) = 1, \tilde\Phi(\ell)=0,
\end{align*}
i.e., the gradient of $\tilde\Phi$ is a constant so that $\tilde\Phi$ is an affine function.
This equation is easily solved and we obtain $\tilde\Phi(x)= \frac{\ell-x}\ell$.  \\
{\rm (2)}  In Lemma \ref{lem:flow}, we are concerned with $q_\ell$ instead of $p_\ell$.
When $d=1$,
\begin{align*}
q_\ell(x) & = \sum_{y\in\T_N^d} p_\ell(x-y)p_\ell(y)
= \frac1{\ell^2} \sum_{1\le x-y \le \ell, 1\le y \le \ell} 1 \\
& = \frac1{\ell^2} \sharp \{ y: 1\le y \le \ell, x-\ell\le y \le x-1\}  \\
& = \left\{ \begin{aligned}
& \frac{x-1}{\ell^2} \quad (\text{if } x-\ell\le 1, \text{ i.e. } x \le \ell+1), \\
& \frac{2\ell+1-x}{\ell^2} \quad (\text{if } x-\ell\ge 1, \text{ i.e. } x \ge \ell+1),
\end{aligned} \right.
\end{align*}
i.e., $q_\ell$ is piecewise affine.  Therefore, its integration $\Phi^\ell$ is piecewise
quadratic.
\end{rem}

Note that
\begin{align*}
(g*p_\ell)(x) &= \sum_{y\in\T_N^d} g_{x-y} p_\ell(y) \\
&= \frac1{|\La_\ell|} \sum_{y\in \La_\ell}g_{x-y} = \overleftarrow{g}_{x,\ell},
\end{align*}
and similarly $(g*\hat p_\ell)(x) = \overrightarrow{g}_{x,\ell}$, where $\hat p_\ell(y) := p_\ell(-y)$.
Therefore,
\begin{align*}
V^\ell & = K \sum_{x\in\T_N^d} \overleftarrow{(\tilde\om_1)}_{x,\ell}
\overrightarrow{(\om_2)}_{x,\ell} \\
& = K \sum_{y\in\T_N^d} \tilde\om_{1,y} (\om_2* q_\ell)(y).
\end{align*}
Accordingly, from Lemma \ref{lem:flow} and $\Phi^\ell(y,y-e_j) = - \Phi^\ell(y-e_j,y)$, one can rewrite
\begin{align*}
V-V^\ell
& = K \sum_{x\in\T_N^d} \tilde\om_{1,x} \{\om_{2,x} - (\om_2 * q_\ell)(x)\} \\
& = K \sum_{x\in\T_N^d} \tilde\om_{1,x} \left\{\om_{2,x} - \sum_{y\in\T_N^d}\om_{2,x-y} q_\ell(y)\right\} \\
& = K \sum_{x\in\T_N^d} \tilde\om_{1,x} \sum_{y\in\T_N^d}\om_{2,x-y} \left\{\de_0(y)- q_\ell(y)\right\} \\
& = K \sum_{x\in\T_N^d} \tilde\om_{1,x} \sum_{y\in\T_N^d}\om_{2,x-y} \sum_{\pm}\sum_{j=1}^d \Phi^\ell(y,y\pm e_j) \\
& = K \sum_{j=1}^d \sum_{x\in\T_N^d} \tilde\om_{1,x} \sum_{y\in\T_N^d} (\om_{2,x-y} -\om_{2,x-y-e_j}) \Phi^\ell(y,y+ e_j) \\
& = K \sum_{j=1}^d \sum_{x\in\T_N^d}  \left(\sum_{y\in\T_N^d} \tilde\om_{1,x+y+e_j} \Phi^\ell(y,y+ e_j) \right)
\{\om_{2,x+e_j} - \om_{2,x} \}.
\end{align*}
Thus, we have shown
\begin{equation}\label{eq:2.13}
V-V^\ell = K \sum_{j=1}^d \sum_{x\in \T_N^d} h_x^{\ell,j} (\om_{2,x+e_j} -\om_{2,x}),
\end{equation}
where
\begin{equation}\label{eq:2.14}
h_x^{\ell,j} \equiv h_x^{\ell,j}(\si_1)
= \sum_{y\in \La_{2\ell-1}} \tilde\om_{1,x+y+e_j}\Phi^\ell(y,y+e_j).
\end{equation}
Note that $h_x^{\ell,j}$ satisfies $h_x^{\ell,j}(\si_1^{x,x+e_j}) = h_x^{\ell,j}(\si_1)$.
This property becomes useful to study the first and second terms of
\eqref{eq:1.2}.  For the third term $V$ of \eqref{eq:1.2}, which we concern now, we will use
the property $h_x^{\ell,j}(\si_2^{x,x+e_j}) = h_x^{\ell,j}(\si_2)$, which is obvious
since $h_x^{\ell,j}$ is a function of $\si_1$, see Lemma \ref{lem:2.13} below.

Another lemma we use is the integration by parts formula under the
Bernoulli measure $\nu_{u_1(\cdot),u_2(\cdot)}$ on $\mathcal{X}_N^2$ with
a spatially dependent mean.  We will apply this formula for the function $h=h_x^{\ell,j}$.
The formula is stated for general $h$ with an error caused by
the non-constant property of $u_2(\cdot)$.

\begin{lem} \label{lem:IP-0} (Integration by parts)
Let $\nu=\nu_{u_1(\cdot),u_2(\cdot)}$ and assume
$e^{-c_1K} \le u_2(x), u_2(y)\le c_2$ holds for  $x, y\in \T_N^d: |x-y|=1$
with some $c_1>0, 0<c_2<1$.
Then, for $h=h(\si_1,\si_2)$ and a probability density $f=f(\si_1,\si_2)$
with respect to $\nu$, we have
$$
\int h(\si_{2,y}-\si_{2,x}) f d\nu=\int h(\si_1,\si_2^{x,y}) \si_{2,x}\big( f(\si_1,\si_2^{x,y})-f(\si_1,\si_2)\big) d\nu +R_1,
$$
and the error term $R_1=R_{1,x,y}$ is bounded as
$$
|R_1| \le C e^{2 c_1K} |\nabla_{x,y}^1 u_2| \int|h(\si_1,\si_2)| f d\nu + \|h-h(\si_1,\si_2^{x,y})\|_\infty,
$$
with some $C=C_{c_2}>0$, where $\nabla_{x,y}^1 u =u(x)-u(y)$.
\end{lem}

\begin{proof}
First we write
$$
\int h(\si_{2,y}-\si_{2,x}) f d\nu=\sum_{\si_1,\si_2} h(\si_1,\si_2) (\si_{2,y}-\si_{2,x}) f(\si_1,\si_2)\nu(\si_1,\si_2).
$$
Then, by a change of variables $\zeta:=\si_2^{x,y}$ and writing $\zeta$ by $\si_2$ again, we have
\begin{align*}
\sum_{\si_2} h(\si_1,\si_2) \si_{2,y} f(\si_1,\si_2)\nu_2(\si_2)
 = \sum_{\si_2} h(\si_1,\si_2^{x,y}) \si_{2,x} f(\si_1,\si_2^{x,y})\nu_2(\si_2^{x,y}),
\end{align*}
where $\nu_2=\nu_{u_2(\cdot)}$ is a probability measure on $\mathcal{X}_N$,
recall $\nu=\nu_{u_1(\cdot)}\otimes \nu_{u_2(\cdot)}$.
To replace the last $\nu_2(\si_2^{x,y})$ by $\nu_2(\si_2)$, we observe
\begin{align*}
\frac{\nu_2(\si^{x,y})}{\nu_2(\si)} & =
1_{\{\si_x=1,\si_y=0\}} \frac{(1-u_2(x))u_2(y)}{u_2(x)(1-u_2(y))}
+ 1_{\{\si_x=0,\si_y=1\}} \frac{u_2(x)(1-u_2(y))}{(1-u_2(x))u_2(y)}
+ 1_{\{\si_x=\si_y\}} \\
& = 1+r_{x,y}(\si),
\end{align*}
with
\begin{align*}
r_{x,y}(\si)= 1_{\{\si_x=1,\si_y=0\}} \frac{u_2(y)-u_2(x)}{u_2(x)(1-u_2(y))}
+ 1_{\{\si_x=0,\si_y=1\}} \frac{u_2(x)-u_2(y)}{(1-u_2(x))u_2(y)}.
\end{align*}
By the condition on $u_2$, this error is bounded as
$$
|r_{xy}(\si)| \le C_0 e^{c_1K} |\nabla_{x,y}^1 u_2|, \quad C_0 = C_{c_2}>0.
$$
These computations are summarized as
\begin{align*}
\int h&(\si_{2,y}-\si_{2,x}) f d\nu=\int h(\si_1,\si_2^{x,y}) \si_{2,x} f(\si_1,\si_2^{x,y}) (1+ r_{xy}(\si_2)) d\nu
- \int h \si_{2,x} fd\nu \\
=&\int h(\si_1,\si_2^{x,y}) \si_{2,x}\big( f(\si_1,\si_2^{x,y})-f(\si_1,\si_2)\big) d\nu  \\
&+ \int(h(\si_1,\si_2^{x,y})-h(\si_1,\si_2))  \si_{2,x} fd\nu
+ \int h(\si_1,\si_2^{x,y}) \si_{2,x} f(\si_1,\si_2^{x,y}) r_{xy}(\si_2) d\nu.
\end{align*}
The second term is bounded by $\|h(\si_1,\si_2^{x,y})-h(\si_1,\si_2)\|_\infty$, since $|\si_{2,x}|\le 1$ and $\int fd\nu=1$.
For the third term denoted by $R_0$, applying the change of variables again,
we have
\begin{align*}
|R_0| & = \left|  \sum_{\si_1,\si_2} h(\si_1,\si_2) \si_{2,y} f(\si_1,\si_2) r_{xy}(\si_2^{x,y}) \nu(\si_1,\si_2^{x,y}) \right| \\
& = \left|  \sum_{\si_1,\si_2} h(\si_1,\si_2) \si_{2,y} f(\si_1,\si_2) r_{xy}(\si_2^{x,y}) \big(1+ r_{xy}(\si_2)\big)
 \nu(\si_1,\si_2)\right| \\
& \le C_0e^{c_1K}  |\nabla_{x,y}^1 u_2| (1+C_0e^{c_1K}  |\nabla_{x,y}^1 u_2| ) \int|h(\si)| fd\nu  \\
& \le C e^{2c_1K} |\nabla_{x,y}^1 u_2| \int|h(\si)| fd\nu,
\end{align*}
since $|\si_{2,y}|\le 1$ and $ |\nabla_{x,y}^1 u_2|\le 2 c_2$.
This completes the proof.
\end{proof}

We apply Lemma \ref{lem:IP-0} to $V-V^\ell$ given in \eqref{eq:2.13}.  Note that
$h_x^{\ell,j}(\si_1)$ is invariant under the transform $\si_2\mapsto \si_2^{x,y}$.
Since we have $\om_{2,x}= \frac{\si_{2,x}-u_2(x)}{\chi(u_2(x))}$ in \eqref{eq:2.13}
instead of $\si_{2,x}$ in Lemma \ref{lem:IP-0}, we need to estimate the error caused by
the $x$-dependence of $\om_{2,x}$ through $u_2(x)$.

\begin{lem} \label{lem:2.13}
We assume that $\nu=\nu_{u_1(\cdot), u_2(\cdot)}$ satisfies the same condition
as in Lemma \ref{lem:IP-0}.  Then, we have
\begin{equation}  \label{eq:2.15}
\int h_x^{\ell,j}(\om_{2,x+e_j}-\om_{2,x}) f d\nu=\int h_x^{\ell,j} \frac{\si_{2,x}}{\chi(u_2(x))}
\big( f(\si_1,\si_2^{x,x+e_j})-f(\si_1,\si_2)\big) d\nu +R_2,
\end{equation}
and the error term $R_2=R_{2,x,j}$ is bounded as
\begin{equation}  \label{eq:4.5}
|R_2| \le Ce^{3c_1K} |\nabla_{x,x+e_j}^1 u_2| \int|h_x^{\ell,j}(\si_1,\si_2)| f d\nu,
\end{equation}
with some $C=C_{c_2}>0$.
\end{lem}

\begin{proof}
By the definition of $\om_x$, denoting $y=x+e_j$, we have
\begin{align*}
\int h_x^{\ell,j}(\om_{2,y}-\om_{2,x}) f d\nu
& =\int h_x^{\ell,j} \left(\frac{\si_{2,y}}{\chi(u_2(y))}-\frac{\si_{2,x}}{\chi(u_2(x))}\right) f d\nu \\
& \qquad - \int h_x^{\ell,j} \left(\frac{u_2(y)}{\chi(u_2(y))}-\frac{u_2(x)}{\chi(u_2(x))}\right) f d\nu \\
& =: I_1 - I_2.
\end{align*}
For $I_2$, we have
\begin{align*}
& \left|\frac{u_2(y)}{\chi(u_2(y))}-\frac{u_2(x)}{\chi(u_2(x))}\right| \\
& \le \frac1{\chi(u_2(x))\chi(u_2(y))} \{\chi(u_2(x)) |u_2(y)-u_2(x)|+|u_2(x)||\chi(u_2(x))-\chi(u_2(y))|\} \\
& \le Ce^{c_1K} |\nabla^1_{x,y}u_2|.
\end{align*}
On the other hand, $I_1$ can be rewritten as
\begin{align*}
I_1
& = \int \frac{h_x^{\ell,j}}{\chi(u_2(x))} (\si_{2,y}-\si_{2,x}) f d\nu
+ \int h_x^{\ell,j} \left(\frac1{\chi(u_2(y))}-\frac1{\chi(u_2(x))}\right) \si_{2,y} f d\nu \\
&=: I_{1,1}+I_{1,2}.
\end{align*}
For $I_{1,1}$, recalling the invariance of $h_x^{\ell,j}$, one can apply Lemma \ref{lem:IP-0} and obtain
$$
I_{1,1} = \frac1{\chi(u_2(x))} \int h_x^{\ell,j} \si_{2,x}\big( f(\si_1,\si_2^{x,y})-f(\si_1,\si_2)\big) d\nu +\frac1{\chi(u_2(x))}R_1.
$$
Finally for $I_{1,2}$,
$$
\left|\frac1{\chi(u_2(y))}-\frac1{\chi(u_2(x))}\right|
= \frac{|\chi(u_2(x))-\chi(u_2(y))|}{\chi(u_2(x))\chi(u_2(y))}
\le Ce^{2c_1K}|\nabla_{x,y}^1u_2|.
$$
Therefore, we obtain the conclusion.
\end{proof}

We can estimate the first term in the right hand side of \eqref{eq:2.15} with $y=x+e_j$
by the Dirichlet form and obtain

\begin{lem}  \label{lem:4.5}
Let $\nu=\nu_{u_1(\cdot), u_2(\cdot)}$ be the Bernoulli measure satisfying the same condition
as in Lemma \ref{lem:IP-0}.  Then, for every $\b>0$, we have
\begin{equation*}
\int h_x^{\ell,j}(\om_{2,x+e_j}-\om_{2,x}) f d\nu
\le \b \mathcal{D}_{x,x+e_j}(\sqrt{f};\nu) + \frac{C}\b e^{3c_1K}
\int (h_x^{\ell,j})^2 f d\nu +R_{2,x,j},
\end{equation*}
where $\mathcal{D}_{x,x+e_j}(\sqrt{f};\nu)$ is a piece of $\mathcal{D}(\sqrt{f};\nu)$
defined on the bond $\{x,x+e_j\}$ and $R_{2,x,j}$ has a bound \eqref{eq:4.5}.
\end{lem}

\begin{proof}
For simplicity, we write $y$ for $x+e_j$.
By decomposing $f(\si_1,\si_2^{x,y})-f(\si_1,\si_2) = \big( \sqrt{f(\si_1,\si_2^{x,y})}+\sqrt{f(\si_1,\si_2)}\big)
\big( \sqrt{f(\si_1,\si_2^{x,y})}-\sqrt{f(\si_1,\si_2)}\big)$, the first term in the right hand side of \eqref{eq:2.15}
can be estimated by
$$
\le \b \mathcal{D}_{x,y}(\sqrt{f};\nu) + \frac{C}{\b \chi(u_2(x))^2}
\int (h_x^{\ell,j})^2 \{f(\si_1,\si_2^{x,y})+f(\si_1,\si_2)\} d\nu.
$$
The integral in the second term divided by $\chi(u_2(x))^2$ is equal to and bounded by
\begin{align*}
 \frac1{\chi(u_2(x))^2}& \int (h_x^{\ell,j})^2 f(\si_1,\si_2) (1+r_{xy}(\si_2)) d\nu  \\
&\le  \frac{1+C_0 e^{c_1K} |\nabla^1_{x,y}u_2|}{\chi(u_2(x))^2}
 \int (h_x^{\ell,j})^2 f d\nu \\
&\le  e^{2c_1K}(1+C_0e^{c_1K}|\nabla_{x,y}^1u_2|) \int (h_x^{\ell,j})^2 f d\nu.
\end{align*}
This shows the conclusion by recalling $|\nabla_{x,y}^1u_2| \le 2c_2$.
\end{proof}

\begin{proof}[Proof of Proposition \ref{prop:2.9}]
Recalling \eqref{eq:2.13} and by Lemma \ref{lem:4.5} taking $\b= \frac{\a N^2}K$
with $\a>0$ sufficiently small, we have
\begin{align*}
\int &(V-V^\ell) f d\nu
= K \sum_{j=1}^d \sum_{x\in \T_N^d} \int h_x^{\ell,j} (\om_{2,x+e_j} -\om_{2,x}) f d\nu\\
& \le \a N^2 \mathcal{D}(\sqrt{f};\nu) + \frac{CK^2}{\a N^2} e^{3c_1K}
\sum_{j=1}^d \sum_{x\in \T_N^d} \int (h_x^{\ell,j})^2 f d\nu
+ K \sum_{j=1}^d \sum_{x\in \T_N^d} R_{2,x,j}.
\end{align*}
For $R_{2,x,j}$, since $|\nabla_{x,x+e_j}^1u_2| \le \frac{CK}N$ from Proposition \ref{prop:2.2},
by \eqref{eq:4.5} estimating $|h_x^{\ell,j}| \le 1+(h_x^{\ell,j})^2$, we have
$$
K |R_{2,x,j}| \le \frac{CK^2}N e^{3c_1K} \int \left(1+(h_x^{\ell,j})^2\right) fd\nu.
$$
Thus, we obtain
\begin{align*}
\int (V-V^\ell) f d\nu
\le&  \a N^2 \mathcal{D}(\sqrt{f};\nu) + \frac{C_\a K^2}{N} e^{3c_1K} \sum_{j=1}^d
\sum_{x\in \T_N^d}  \int (h_x^{\ell,j})^2 f d\nu
 + CK^2e^{3c_1K}N^{d-1}.
\end{align*}

For the second term, we first decompose the sum $\sum_{x\in \T_N^d}$ as
$\sum_{y\in \La_{2\ell}}\sum_{z\in (4\ell) \T_N^d}$ and apply the entropy inequality
noting that the variables $\{h_x^{\ell, j}\}$ are
$(2\ell-1)$-dependent:
\begin{align*}
\sum_{x\in \T_N^d}  \int (h_x^{\ell,j})^2 f d\nu
& \le \frac1\ga \sum_{y\in \La_{2\ell}} \left( H(\mu|\nu) +
\log \int \exp\left\{ \ga \sum_{z \in (4\ell)\T^d_N} (h_{z+y}^{\ell,j})^2 \right\} d\nu\right) \\
& = \frac1\ga (4\ell)^d \left( H(\mu|\nu) + \sum_{z \in (4\ell)\T^d_N}
\log\int \exp\left\{ \ga (h_{z+y}^{\ell,j})^2 \right\} d\nu \right).
\end{align*}
However, since $h_x^{\ell,j}$ is a weighted sum of independent
random variables, by applying Lemma \ref{lem:ci} (concentration inequality)
stated below, we have
$$
\log \int e^{\ga(h_x^{\ell, j})^2} d\nu \le 2
$$
for every $0<\ga \le \frac{C_0}{\si^2}$, where $C_0$ is a universal constant
and $\si^2$ is the supremum of the variances of $h_x^{\ell, j}$.  By Lemma
\ref{lem:flow},
$$
\si^2 \le C_d g_d(\ell).
$$
Therefore, we have
\begin{align*}
\sum_{x\in \T_N^d}  \int (h_x^{\ell,j})^2 f d\nu
 \le \frac1\ga (4\ell)^d \left( H(\mu|\nu) + 2 (\tfrac{N}{4\ell})^d  \right).
\end{align*}

Thus, choosing $\frac1\ga=\frac{C_d}{C_0}g_d(\ell)$, we have shown
\begin{align*}
\int (V-V^\ell) f d\nu
\le \a N^2 \mathcal{D}(\sqrt{f};\nu) +
\frac{\bar C_\a \ell^d g_d(\ell) K^2e^{3c_1K}}{N}
\left( H(\mu|\nu) + \frac{N^d}{\ell^d} \right)
+ CK^2e^{3c_1K}N^{d-1}.
\end{align*}
Now recall $1\le K \le \de \log N$ so that $e^{3c_1K} \le N^{3c_1\de}$ and choose
$\de>0$ such that $3c_1\de \le \k$ for a given small $\k>0$.
Choose $\ell = N^{\frac1d-\k}$ when $d\ge 2$ and $N^{\frac12-\k}$ when $d=1$.
Then, when $d\ge 2$, we have
$$
\frac{\ell^d g_d(\ell)K^2}{N}e^{3c_1K} \le CN^{-\k (d-1)} (\log N)^3\le1, \quad
\frac{N^d}{\ell^d} =N^{d-1+d\k},
\quad K^2e^{3c_1K}N^{d-1} \le N^{d-1+2\k},
$$
which shows \eqref{eq:2.11}.  When $d=1$,
$$
\frac{\ell^2K^2}{N}e^{3c_1K} \le CN^{-\k} (\log N)^2\le1, \quad
\frac{N}{\ell} =N^{\frac12+ \k},
\quad K^2 e^{3c_1K}N^{d-1} \le N^{2\k}.
$$
This shows the conclusion for $d=1$.
\end{proof}

\begin{lem} \label{lem:ci} (concentration inequality)
Let $\{X_i\}_{i=1}^n$ be independent random variables with values in the
intervals $[a_i,b_i]$.  Set $\bar{X}_i = X_i - E[X_i]$ and $\k = \sum_{i=1}^n (b_i-a_i)^2$.
Then, for every $\ga\in [0,\k^{-1}]$, we have
$$
\log E\left[\exp \left\{\ga \left(\sum_{i=1}^n \bar{X}_i\right)^2\right\} \right]
\le 2 \ga\k.
$$
\end{lem}

The second step is to estimate the integral $\int V^\ell f d\nu$, where
$V^\ell$ is given by \eqref{eq:4.VL}.

\begin{prop}  \label{prop:4.7}
We assume the same conditions as Proposition \ref{prop:2.9}.
Then, for $\k>0$,  we have
\begin{equation}  \label{eq:4.prop47}
\int V^\ell f d\nu \le CK H(\mu_t|\nu_t) + C_\k N^{d-1+\k},
\end{equation}
with some $C_\k>0$ when $d\ge 2$.  When $d=1$, the last term is replaced
by $C_\k N^{\frac12+\k}$.
\end{prop}

\begin{proof}
We again decompose the sum $\sum_{x\in \T_N^d}$ in \eqref{eq:4.VL} as
$\sum_{y\in \La_{2\ell}}\sum_{z\in (4\ell) \T_N^d}$, and then, noting the $(2\ell)$-dependence
of $\overleftarrow{(\tilde\om_1)}_{x,\ell} \overrightarrow{(\om_2)}_{x,\ell}$,
use the entropy inequality and the concentration inequality to show
\begin{align*}
\int V^\ell f d\nu
& \le \frac{K}\ga \sum_{y\in\La_{2\ell}} \left\{ H(\mu_t|\nu_t) +
\sum_{z\in (4\ell)\T_N^d} \log E^{\nu_t}[e^{\ga \overleftarrow{(\tilde\om_1)}_{z+y,\ell}
\overrightarrow{(\om_2)}_{z+y,\ell}}] \right\}  \\
& \le \frac{K(4\ell)^d}\ga \left\{ H(\mu_t|\nu_t) + \frac{N^d}{(4\ell)^d} C_1 \ga \ell^{-d}\right\},
\end{align*}
for $\ga=c\ell^d$ with $c>0$ small enough.
Note that, by the central limit theorem, $\overleftarrow{(\tilde\om_1)}_{x,\ell},
\overrightarrow{(\om_2)}_{x,\ell}$ are close to $C_2\ell^{-d/2}N(0,1)$ in law for large $\ell$,
respectively.  Since $\ell = N^{\frac1d-\k}$ when $d\ge 2$, we have $\frac{KN^d}{\ell^d}
\le N^{d-1+d\k} \cdot \de \log N$ and obtain \eqref{eq:4.prop47}.  When $d=1$,
since $\ell = N^{\frac12-\k}$, we have $\frac{KN^d}{\ell^d} = N^{\frac12 +\k} \de \log N$.
\end{proof}

\subsection{Proof of \eqref{2.10}}  \label{sec:4.2}

We now discuss the contribution of
$$
V_1 :=
-\frac{N^2}2 \sum_{x,y\in\T_N^d: |x-y|=1} (u_1(y)-u_1(x))^2\om_{1,x} \om_{1,y}
$$
in \eqref{eq:1.2}, which arises from the Kawasaki part.
The second term $V_2$ can be treated similarly.
We may think $N^2(u_1(y)-u_1(x))^2$ as if $K$
in the argument we have developed.
However, from Proposition \ref{prop:2.2}, we have
\begin{equation}  \label{eq:N2u}
N^2(u_1(y)-u_1(x))^2 \le C K^2.
\end{equation}
This means that we may replace $K$ by $K^2$ properly in the estimates obtained
in Propositions \ref{prop:2.9} and \ref{prop:4.7} for the first and second terms.
Since $K^2=\de^2\log N$ appearing in the error terms
can be absorbed by $N^\k$ for every $\k>0$,
this leads to
\begin{equation}  \label{eq:4.1+2}
\int(V_1+V_2) d\mu_t
\le \a N^2 \mathcal{D}\left( \sqrt{\tfrac{d\mu_t}{d\nu_t}}; \nu_t\right)
+ C_\k K^2 H(\mu_t|\nu_t) + C_{\a,\k} N^{d-1+\k},
\end{equation}
for every $a, \k>0$, when $d\ge 2$ and the last term is replaced by
$C_{\a,\k}N^{\frac12+\k}$  when $d=1$.

\section{Consequence of Theorem \ref{thm:EstHent}}  \label{sec:4}

Recall that $\mu_t^N$ is the distribution
of $\tilde\si^N(t)$ on $\mathcal{X}_N^2$ and
$u_i^N(t) = \{u_i^N(t,x)\}_{x\in \T_N^d}$, $i=1,2$
is the solution of the discretized hydrodynamic equation
\eqref{eq:HD-discre}.  The Bernoulli measure on $\mathcal{X}_N^2$
with mean $\{u_i^N(t,x)\}_{x\in \T_N^d}$ is denoted by $\nu_t^N$.
Then Theorem \ref{thm:EstHent} shows
$H(\mu_t^N|\nu_t^N) =o(N^d)$ under a proper choice of $K=K(N)
\nearrow \infty$.  We define macroscopic functions $u_i^N(t,r), r \in \T^d$
as step functions
\begin{equation}  \label{eq:4.10}
u_i^N(t,r) = \sum_{x\in\T_N^d} u_i^N(t,x) 1_{B(\frac{x}N,\frac1N)}(r),
\quad r \in \T^d,
\end{equation}
from the microscopic functions $u_i^N(t,x), x\in \T_N^d$,
where $B(\frac{x}N,\frac1N) = \prod_{j=1}^d [\frac{x_j}N-\frac1{2N},
\frac{x_j}N+\frac1{2N})$ is the box with center $\frac{x}N$ and side length $\frac1N$.

Under our choice of $K$, the entropy inequality
$$
\mu_t^N(A) \le \frac{\log 2 + H(\mu_t^N|\nu_t^N)}{\log \{1+1/\nu_t^N(A)\}}
$$
combined with
Proposition \ref{prop:4.4-1} stated below shows that
\begin{equation}  \label{eq:4.18}
\lim_{N\to\infty}\mu_t^N ({\cal A}_{N,t}^\e)=0,
\end{equation}
for every $\e>0$, where
\[
{\cal A}_{N,t}^\e \equiv
{\cal A}_{N,t,\fa}^\e := \left\{\eta\in{\cal X}_N \,; ~\left|
\lan \a_i^N,\fa\ran   - \lan u_i^N(t,\cdot),\fa\ran\right| > \e, \, i=1,2\right\},
\quad \fa \in C^\infty(\T^d).
\]

\begin{prop}  \label{prop:4.4-1}
There exists $C=C_{\e,\fa}>0$ such that
$$
\nu_t^N ({\cal A}_{N,t}^\e) \le e^{-CN^d}.
$$
\end{prop}

\begin{proof}
Since
$$
X_i:= \lan \a_i^N,\fa\ran   - \lan u_i^N(t,\cdot),\fa\ran
= \frac1{N^d} \sum_{x\in\T_N^d} \left\{\si_{i,x}-u_i^N(t,\frac{x}N)\right\}\fa(\frac{x}N) +o(1),
$$
for $\fa \in C^\infty(\T^d)$, we have
\begin{align*}
\nu_t^N({\cal A}_{N,t}^\e)
& \le e^{-\ga \e N^d} E^{\nu_t^N}[ e^{\ga N^d |X_i|}] \\
& \le e^{-\ga \e N^d} \left\{E^{\nu_t^N}[ e^{\ga N^d X_i}] + E^{\nu_t^N}[ e^{-\ga N^d X_i}] \right\},
\end{align*}
for every $\ga>0$.  However, by the independence of $\si_{i,x}$
under $\nu_t^N$, we have
\begin{align*}
E^{\nu_t^N}[ e^{\pm \ga N^d X_i}]
& = \prod_{x\in\T_N^d} E^{\nu_t^N}[ e^{\pm \ga \{\si_{i,x}-u_{i,x}\}\fa_x+o(1)}] \\
&= \prod_{x\in \T_N^d} \left\{ e^{\pm \ga(1-u_{i,x})\fa_x} u_{i,x}
+ e^{\mp \ga u_{i,x} \fa_x}(1-u_{i,x})\right\}+o(1),
\end{align*}
where $u_{i,x} = u_i^N(t,x)$ and $\fa_x = \fa(\frac{x}N)$.  However, by the Taylor's
formula applied at $\ga=0$, we see
$$
\left| e^{\pm \ga(1-u_{i,x})\fa_x} u_{i,x} + e^{\mp \ga u_{i,x} \fa_x}(1-u_{i,x}) - 1 \right|
\le \frac{\ga^2}2 C, \quad C=C_{\|\fa\|_\infty},
$$
for $0<\ga\le 1$.  Thus we obtain
$$
\nu_t^N({\cal A}_{N,t}^\e)\le e^{-\ga \e N^d + C\ga^2 N^d},
$$
for $\ga>0$ sufficiently small.  This shows the conclusion.
\end{proof}

\section{Convergence of the solution of the discretized hydrodynamic
equation to that of the free boundary problem}  \label{sec:5}

We show $u_i^N(t,r), t\in [0,T], r\in \T^d, i=1, 2$ appearing in
\eqref{eq:4.18}, which is defined by \eqref{eq:4.10}
from the solution of the discretized hydrodynamic
equation \eqref{eq:HD-discre}, converges to the unique weak solution of
the free boundary problem \eqref{eq:P}.
This can be done along with \cite{CDHMN}, in a discrete setting.
Once this is shown, combined with \eqref{eq:4.18}, the proof of
Theorem \ref{thm:1.1} is complete.

\begin{lem}
$$
\int_0^T\int_{\T^d} u_1^N(t,r)u_2^N(t,r) dtdr \le \frac1K.
$$
\end{lem}

\begin{proof} (cf.\ Lemma 2.3 of \cite{CDHMN} with $\fa\equiv 1$)
From \eqref{eq:HD-discre}, we have
\begin{align*}
K& \sum_{x\in \T_N^d} \int_0^T u_1^N(t,x)u_2^N(t,x)dt  \\
& = d_1 \sum_{x\in \T_N^d} \int_0^T \De^N u_1^N(t,x)dt
+ \sum_{x\in \T_N^d} u_1^N(0,x) - \sum_{x\in \T_N^d} u_1^N(T,x)
\le N^d,
\end{align*}
which implies the conclusion.
\end{proof}

\begin{lem}
$$
\int_0^T\int_{\T^d} |\nabla^N u_i^N(t,r)|^2 dtdr \le \frac1{2d_i}, \quad i=1,2,
$$
where $\nabla^Nu(r) = \{N(u(r+\frac1N e_j) - u(r))\}_{j=1}^d$.
\end{lem}

\begin{proof} (cf.\ Lemma 2.4 of \cite{CDHMN} with $\fa\equiv 1$)
From \eqref{eq:HD-discre}, we have
\begin{align*}
\frac12\frac{d}{dt} \int_{\T^d}  u_1^N(t,r)^2 dr
+ d_1 \int_{\T^d}  |\nabla^N u_1^N(t,r)|^2 dr
= -K \int_{\T^d}  u_1^N(t,r)^2u_2^N(t,r) dr \le 0,
\end{align*}
and this implies
\begin{align*}
d_1 \int_0^T dt \int_{\T^d}  |\nabla^N u_1^N(t,r)|^2 dr
\le \frac12 \int_{\T^d}  \left\{u_1^N(0,r)^2- u_1^N(T,r)^2\right\} dr \le \frac12.
\end{align*}
The proof for $u_2^N$ is similar.
\end{proof}

These two lemmas with the help of Fr\'echet-Kolmogorov theorem show that
$\{u_i^N(t,r)\}_N$ are relatively compact in $L^2([0,T]\times\T^d)$.
In fact, two lemmas prove the equi-continuity of $\{u_i^N(t,r)\}_N$
in the space $L^2([0,T]\times\T^d)$ as in Lemmas 2.6 and 2.7 of
\cite{CDHMN}.

\begin{cor} \label{cor:5.3}
(cf.\ Corollary 3.1 of \cite{CDHMN})
From any subsequence of  $\{u_i^{N}(t,r)\}_N$, $i=1,2$,
one can find further subsequences $\{u_i^{N_k}(t,r)\}_k$, $i=1,2$, and
$u_i \in L^2([0,T]\times\T^d)$, $i=1,2$ such that
$$
u_i^{N_k} \to u_i \quad \text{ strongly in } L^2([0,T]\times\T^d)
\text{ and a.e.\ in } [0,T]\times\T^d
$$
as $k\to\infty$.
\end{cor}

\begin{lem} \label{lem:5.4}
(cf.\ Lemma 3.2 of \cite{CDHMN})
$u_1 u_2 =0$ a.e.\ in $[0,T]\times\T^d$.
\end{lem}

Set
$$
w^N:= u_1^N-u_2^N \quad \text{ and } \quad w:=u_1-u_2.
$$
From Corollary \ref{cor:5.3} and Lemma \ref{lem:5.4}, $w^{N_k}\to w$ strongly in
$L^2([0,T]\times\T^d)$ and a.e.\ in $[0,T]\times\T^d$
as $k\to\infty$ and furthermore
$$
u_1=w^+ \quad \text{ and } \quad u_2=w^-.
$$

\begin{prop}
$w$ is the unique weak solution of \eqref{eq:P}.
\end{prop}

\begin{proof}
It is sufficient to check the property (ii) of Definition
\ref{def1.1} for $w$.  From \eqref{eq:HD-discre}, for $\psi\in
C^{1,2}([0,T]\times \T^d)$ such that $\psi(T,0)=0$,
\begin{align*}
\int_0^T \int_{\T^d} &(u_1^N(t,r)-u_2^N(t,r)) \partial_t\psi(t,r) dr dt
- \int_{\T^d}(u_1^N(0,r)-u_2^N(0,r)) \psi(0,r) dr \\
& = \int_0^T \int_{\T^d}(d_1 u_1^N(t,r)-d_2 u_2^N(t,r)) \De^N \psi(t,r) dr dt.
\end{align*}
We obtain the property (ii) for $w$ by passing to the limit $k\to\infty$
along with the subsequence $N=N_k$.
\end{proof}

Because of the uniqueness of $w$, without taking subsequences,
$u_i^N(t,r)$, $i=1,2$ themselves converge to $u_i(t,r)$
strongly in $L^2([0,T]\times\T^d)$ and a.e.\ in $[0,T]\times\T^d$
as $N\to\infty$.  This combined with \eqref{eq:4.18}
completes the proof of Theorem \ref{thm:1.1}.

\bigskip

\noindent {\bf Acknowledgments:}
T. Funaki is supported in part by JSPS KAKENHI, Grant-in-Aid
for Scientific Researches (A) 18H03672 and (S) 16H06338.
E. Presutti thanks the GSSI. 
M.~E.~Vares acknowledges support of CNPq (grant 305075/2016-0)
and FAPERJ (grant E-26/203.048/2016).

\medskip

\noindent Anna De Masi \\
Dipartimento di Ingegneria e Scienze dell'Informazione e Matematica, \\
Universit\`a degli studi dell'Aquila, L'Aquila, 67100 Italy\\
e-mail: demasi@univaq.it

\medskip

\noindent Tadahisa Funaki \\
Department of Mathematics, School of Fundamental Science and Engineering, \\
Waseda University, 3-4-1 Okubo, Shinjuku-ku, Tokyo 169-8555, Japan. \\
e-mail: funaki@ms.u-tokyo.ac.jp
\medskip

\noindent Errico Presutti \\
GSSI, viale F. Crispi 7, 67100 L'Aquila, Italy\\
e-mail: errico.presutti@gmail.com
\medskip

\noindent Maria Eulalia Vares \\
Instituto de Matem\'atica. Universidade Federal do Rio de Janeiro, \\
Centro de Tecnologia - Bloco C, Av. Athos da Silveira Ramos, 149, \\
Ilha do Fund\~ao, 21941--909, Rio de Janeiro, RJ, Brazil \\
e-mail: eulalia@im.ufrj.br


\begin{thebibliography}{99}

\bibitem{CDHMN}{\sc E.C.M.\ Crooks, E.N.\ Dancer, D.\ Hilhorst, M.\ Mimura
and H.\ Ninomiya},
{\it Spatial segregation limit of a competition-diffusion system with
Dirichlet boundary conditions}, Nonlinear Anal.\ Real World Appl.,
{\bf 5} (2004), 645--665.

\bibitem{DHMP}{\sc E.N.\ Dancer, D.\ Hilhorst, M.\ Mimura and L.A.\ Peletier},
{\it Spatial segregation limit of a competition-diffusion system}.
European J.\ Appl.\ Math., {\bf 10} (1999), 97--115.

\bibitem{DDJ}{\sc E.\ Daus, L.\ Desvillettes and A.\ J\"ungel}
{\it Cross-diffusion systems and fast-reaction limits}, arXiv:1710.03590.

\bibitem{DD}{\sc T.\ Delmotte and J.-D.\ Deuschel},
{\it On estimating the derivatives of symmetric diffusions in stationary
random environment, with applications to $\nabla\phi$ interface model},
Probab.\ Theory Related Fields, {\bf 133} (2005), 358--390.

\bibitem{F99}{\sc T.\ Funaki},
{\it Free boundary problem from stochastic lattice gas model},
Ann.\ Inst.\ H.\ Poincar\'e, Probab.\ Statist., {\bf 35} (1999), 573--603.

\bibitem{F18}{\sc T.\ Funaki},
{\it Hydrodynamic limit for exclusion processes},
Comm.\ Math.\ Statist., {\bf 6} (2018), 417--480.

\bibitem{FT}{\sc T.\ Funaki and K.\ Tsunoda},
{\it Motion by mean curvature from Glauber-Kawasaki dynamics},
arXiv:1812.10182.

\bibitem{IMMN}{\sc M.\ Iida, H.\ Monobe, H.\ Murakawa and H.\ Ninomiya},
{\it  Vanishing, moving and immovable interfaces in fast reaction limits},
J.\ Differential Equations, {\bf 263} (2017), 2715--2735.

\bibitem{JM}{\sc M.\ Jara and O.\ Menezes},
{\it Non-equilibrium fluctuations of interacting particle systems},
arXiv:1810.09526.

\bibitem{Sa}{\sc M.\ Sasada},
{\it Hydrodynamic limit for two-species exclusion processes.},
Stoch.\ Proc.\ Appl., {\bf 120} (2010), 494--521.

\bibitem{SZ}{\sc D.W.\ Stroock and W.\ Zheng},
{\it Markov chain approximations to symmetric diffusions},
Ann.\ Inst.\ H.\ Poincare Probab.\ Statist., {\bf 33} (1997), 619--649.

\end{thebibliography}
\end{document}